\documentclass[12pt,reqno]{amsart}

\usepackage{bbm,hyperref,setspace}
\usepackage{fullpage}

\numberwithin{equation}{section}

\theoremstyle{theorem}
\newtheorem{thm}{Theorem}[section]
\newtheorem{prp}{Proposition}[section]
\newtheorem{cor}[thm]{Corollary}

\newcommand{\e}{\varepsilon}
\newcommand{\R}{\mathbb R}
\newcommand{\E}{\mathbb E}
\newcommand{\p}{\mathbb P}
\newcommand{\dfn}{\stackrel{\mathrm{def}}{=}}

\begin{document}

\title{A simple observation on random matrices with continuous diagonal entries}

\author{Omer Friedland}
\address{Institut de Math\'ematiques de Jussieu, Universit\'e Pierre et Marie Curie, 4 Rue Jussieu, Paris, 75005, France}
\email{friedland@math.jussieu.fr}

\author{Ohad Giladi}
\address{Department of Mathematical and Statistical Sciences, University of Alberta, Edmonton, AB T6G 2G1, Canada}
\email{giladi@ualberta.ca}

\subjclass[2010]{60B20,15B52}

\date{\today}

\maketitle

\begin{abstract}
Let $T$ be an $n\times n$ random matrix, such that each diagonal entry $T_{i, i}$ is a continuous random variable, independent from all the other entries of $T$. Then for every $n\times n$ matrix $A$ and every $t\ge0$
$$
\p\Big[|\det(A+T)|^{1/n}\le t\Big]\le2bnt, 
$$
where $b>0$ is a uniform upper bound on the densities of $T_{i, i}$.
\end{abstract}

\section{introduction}

In this note we are interested in the following question: Given an $n\times n$ random matrix $T$, what is the probability that $T$ is invertible, or at least ``close'' to being invertible? One natural way to measure this property is to estimate the following small ball probability
$$
\p\Big[s_n(T)\le t\Big], 
$$
where $s_n(T)$ is the smallest singular value of $T$, 
$$
s_n(T)\dfn\inf_{\|x\|_2=1}\|Tx\|_2=\frac{1}{\|T^{-1}\|}.
$$

\smallskip

In the case when the entries of $T$ are i.i.d random variables with appropriate moment assumption, the problem was studied in \cite{BVW, R, RV, TV2, TV3}. We also refer the reader to the survey \cite{NV}. In particular, in \cite{RV} it is shown that if the entries of $T$ are i.i.d subgaussian random variables, then
\begin{align}\label{RV bound}
\p\Big[s_n(T)\le t\Big]\le C\sqrt nt+e^{-cn}, 
\end{align} 
where $c, C$ depend on the moment of the entries.

\smallskip

Several cases of dependent entries have also been studied. A bound similar to \eqref{RV bound} for the case when the rows are independent log-concave random vectors was obtained in \cite{AGLPT1, AGLPT2}. Another case of dependent entries is when the matrix is symmetric, which was studied in \cite{ESY, FV, MS, N1, N2, V}. In particular, in \cite{ESY} it is shown that if the above diagonal entries of $T$ are continuous and satisfy certain regularity conditions, namely that the entries are i.i.d subgaussian and satisfy certain smoothness conditions, then
$$
\p\Big[s_n(T)\le t\Big]\le C\sqrt nt.
$$

\smallskip

The regularity assumptions were completely removed in \cite{FV} at the cost of a $n^{5/2}$ (independence of the entries in the non-symmetric part is still needed). On the other hand, in the discrete case, the result of \cite{V} shows that if $T$ is, say, symmetric whose above diagonal entries are i.i.d Bernoulli random variables, then
$$
\p\Big[ s_n(T) =0 \Big] \le e^{-n^c}, 
$$
where $c$ is an absolute constant.

\smallskip

A more general case is the so called \emph{Smooth Analysis} of random matrices, where now we replace the matrix $T$ by $A+T$, where $A$ being an arbitrary deterministic matrix. The first result in this direction can be found in \cite{SST}, where it is shown that if $T$ is a random matrix with i.i.d standard normal entries, then
\begin{align}\label{Smooth bound}
\p\Big[s_n(A+T)\le t\Big]\le C\sqrt nt.
\end{align}

\smallskip

Further development in this direction can be found in \cite{TV4}, where estimates similar to \eqref{Smooth bound} are given in the case when $T$ is a Bernoulli random matrix, and in \cite{FV, N1, N2}, where $T$ is symmetric.

\smallskip

An alternative way to measure the invertibility of a random matrix $T$ is to estimate $\det(T)$, which was studied in \cite{CV, TV1, TV092} (when the entries are discrete distributions). Here we show that if the diagonal entries are independent continuous random variables, we can easily get a small ball estimate for $\det(A+T)$, where $A$ being an arbitrary deterministic matrix.

\begin{thm}\label{main thm}
Let $T$ be an $n \times n$ random matrix, such that each diagonal entry $T_{i, i}$ is a continuos random variable, independent from all the other entries of $T$. Then for every $n\times n$ matrix $A$ and every $t\ge 0$
$$
\p \Big[ |\det (A+T) |^{1/n} \le t \Big] \le 2 b n t, 
$$
where $b>0$ is a uniform upper bound on the densities of $T_{i, i}$.
\end{thm}

We remark that the proof works if we replace the determinant by the permanent of the matrix (see \cite{CV} for the difference between the notions).

\smallskip

Now, we use Theorem \ref{main thm} to get a small ball estimate on the norm and smallest singular value of a random matrix.

\begin{cor}\label{main cor}
Let $T$ be a random matrix as in Theorem \ref{main thm}. Then
\begin{align}\label{bound norm}
\p\Big[\|T\| \le t\Big] \le 2b nt, 
\end{align} 
and
\begin{align}\label{bound sing value}
\p\Big[s_n(T) \le t\Big] \le (2b)^{\frac{n}{2n-1}}\left(\E \|T\|\right)^{\frac{n-1}{2n-1}}t^{\frac{1}{2n-1}}.
\end{align}
\end{cor}

Corollary \ref{main cor} can be applied to the case when the random matrix $T$ is symmetric, under very weak assumptions on the distributions and the moments of the entries and under {\it no independence} assumptions on the above diagonal entries. Note that in this case when $T$ is symmetric, we have 
$$
\|T\| = \sup_{\|x\|_2=1}\langle Tx, x\rangle \ge \max_{1\le i \le n} |T_{i, i}|.
$$

Thus, in this case we get a far better small ball estimate for the norm
$$
\p \Big[ \|T\| \le t\Big] \le (2bt)^n.
$$ 

\smallskip

Finally, in Section \ref{section twobytwo} we show that in the case of $2\times 2$ matrices, we use an ad-hoc argument to obtain a better bound than the one obtained in Theorem \ref{main thm}. We do not know what is the right order when the dimension is higher.
 
\section{Proof of Theorem \ref{main thm}}

Before we give the proof of Theorem \ref{main thm}, we fix some notation. First, let $M=A+T$, and let $M_k$ be the matrix $M$ after erasing the last $n-k$ rows and last $n-k$ columns. Also, let $\Omega_k$ be the $\sigma$-algebra generated by the entries of $M_k$ \emph{except} $M_{k, k}$.

\begin{proof}[Proof of Theorem \ref{main thm}]
We have
$$
|\det (M_k)| = \Big|M_{k, k} \det(M_{k-1}) + f_k\Big|, 
$$
where $f_k$ is measurable with respect to $\Omega_k$. We also have
\begin{align*}
\p & \Big[ |\det (M_k) |\le \e_k \Big] \\ 
& \le \p \Big[ |\det (M_k) | \le \e_k \wedge |\det (M_{k-1}) | \ge \e_{k-1} \Big] + \p \Big[ |\det(M_{k-1})| \le \e_{k-1} \Big] .
\end{align*}

\smallskip

Now, 
\begin{align*}
\p & \Big[ |\det (M_k) | \le \e_n \wedge |\det(T_{k-1}| \ge \e_{k-1} \Big] \\ 
& = \E \Bigg[ \p \left[ |M_{k, k} \det(M_{k-1}) + f_k | \le \e_k \Big| \Omega_k \right] \cdot {\mathbbm 1}_{\{|\det(M_{k-1})| \ge \e_{k-1} \}} \Bigg] \\ 
& \le \sup_{\gamma \in \R} \p \left[ | M_{k, k} + \gamma | \le \frac{\e_k}{\e_{k-1}} \right] \le 2 b \frac{\e_k}{\e_{k-1}}, 
\end{align*}
where the last inequality follows from the fact for a continuous random variable $X$ we always have 
\begin{align}\label{bdd}
\sup_{\gamma\in \R }\p\Big[|X+\gamma|\le t\Big] \le 2bt, 
\end{align}
where $b>0$ is an upper bound on the density of $X$.

\smallskip

Thus, we get
$$
\p \Big[ |\det (M_k) |\le \e_k \Big] \le 2 b \frac{\e_k}{\e_{k-1}} + \p \Big[ |\det(M_{k-1})| \le \e_{k-1} \Big], 
$$
Also, note that 
$$
\p \Big[ |\det (M_1) | \le \e_1 \Big] = \p \Big[ |T_{1, 1}+A_{1, 1}| \le \e_1 \Big] \stackrel{\eqref{bdd}}{\le} 2b\e_1 .
$$

\smallskip

Therefore, 
$$
\p \Big[ |\det (M_n) |\le \e_n \Big] \le 2b\left[\e_1+\sum_{k=2}^n\frac{\e_k}{\e_{k-1}}\right].
$$

\smallskip

Choosing $\e_j = t^{j/n}$, the result follows.
\end{proof}
 
\smallskip
 
Corollary \ref{main cor} now follows immediately.
\begin{proof}[Proof of Corollary \ref{main cor}]
Let $s_1(T) \ge \dots \ge s_n(T)$ be the singular values of $T$. We have
$$
|\det (T)| = \prod_{i=1}^n s_i(T) \le \left(s_1(T)\right)^n .
$$

\smallskip

Thus, by Theorem \ref{main thm}, 
$$
\p \Big[ s_1(T) \le t \Big] \le \p \Big[ |\det (T)|^{1/n} \le t \Big] \le 2bnt, 
$$
which proves \eqref{bound norm}.

\smallskip

To prove \eqref{bound sing value}, note that 
\begin{align}\label{bound with beta}
|\det(T)| = \prod_{i=1}^n s_i(T) \le s_1(T)^{n-1}s_n(T) \le \|T\|^{n-1}s_n(T).
\end{align}

\smallskip

Thus, 
\begin{align}\label{splitting}
\p\Big[s_n(T) \le t\Big] \le \p\Big[s_n(T) \le t \wedge \|T\|\le \beta\Big] + \p\Big[\|T\| > \beta\Big] 
\end{align}

\smallskip

For the first term, we have by \eqref{bound with beta} and Theorem \ref{main thm}, 
$$
\p\Big[s_n(T) \le t \wedge \|T\|\le \beta\Big] \le \p\Big[\det(T) \le \beta^{n-1}t\Big] \le 2b\beta^{\frac {n-1} n}t^{1/n}.
$$

\smallskip

Also, 
\begin{align}\label{cheb}
\p\Big[\|T\| > \beta\Big] \le \frac{\E \|T\|}{\beta}.
\end{align}
 
 \smallskip
 
Thus, by \eqref{splitting} and \eqref{cheb}, 
$$
\p\Big[s_n(T) \le t\Big] \le 2b\beta^{\frac {n-1} n}t^{1/n} + \frac{\E \|T\|}{\beta}.
$$

\smallskip

Optimizing over $\beta$ gives \eqref{bound sing value}.
\end{proof}

\section{The case of $2\times 2$ matrices}\label{section twobytwo}

As discussed in the introduction, we show that for $2\times 2$ matrices the small ball estimate on the determinant obtained in Theorem \ref{main thm} is not sharp. To do that, we use the well known fact that if $X$ and $Y$ are continuous random variables with joint density function $f_{X, Y}(\cdot, \cdot)$ then $X\cdot Y$ has a density function which is given by
$$
f_{X\cdot Y}(z) = \int_{-\infty}^{\infty}f_{X, Y}\left(w, \frac z w\right)\frac{dw}{|w|}, 
$$
where $f_X$, $f_Y$ are the density functions of $X$, $Y$, respectively.

\smallskip

We thus have the following.

\begin{prp}\label{prop prod}
Assume that $X$ and $Y$ are independent continuous random variables, with $f_X \le b$, $f_Y\le b$. Then $f_{X\cdot Y}$, the density function of $X\cdot Y$ satisfies
\begin{align*}
f_{X\cdot Y}(z) \le \begin{cases} 2b+2b^2|\log (|z|)| & |z| \le 1, \\ 2b & |z| \ge 1.\end{cases}
\end{align*}
\end{prp}

\begin{proof}
Assume first that $|z| \le 1$. Write
\begin{align}\label{int splitting z small}
& \nonumber f_{X\cdot Y}(z) = \int_{-\infty}^{\infty}f_{X, Y}\left(w, \frac z w\right)\frac{dw}{|w|}
\\ & = \int_{|w| \le |z| }f_{X, Y}\left(w, \frac z w\right)\frac{dw}{|w|}+ \int_{|z|\le |w| \le 1}f_{X, Y}\left(w, \frac z w\right)\frac{dw}{|w|}+ \int_{|w|\ge 1}f_{X, Y}\left(w, \frac z w\right)\frac{dw}{|w|}.
\end{align}

\smallskip

Since $X$ and $Y$ are independent, $f_{X, Y}(x, y)= f_X(x)\cdot f_Y(y)$. We estimate each term of \eqref{int splitting z small} separately. Assume first that $|z| \le 1$
\begin{align}
& \label{1st int}\int_{|w| \le |z| }f_X(w)\cdot f_{Y}\left(\frac z w\right)\frac{dw}{|w|} \le b\int_{|w| \le |z| }f_{Y}\left(\frac z w\right)\frac{dw}{|w|} = b \int_{|y|\ge 1}f_Y(y)\frac{dy}{|y|} \le b \\
& \label{2nd int}\int_{|z|\le |w| \le 1}f_X(w)\cdot f_{Y}\left(\frac z w\right)\frac{dw}{|w|} \le b^2\int_{|z|\le |w| \le 1}\frac{dw}{|w|} = 2b^2|\log (|z|)| \\ 
& \label{3rd int} \int_{|w| \ge 1 }f_X(w)\cdot f_{Y}\left(\frac z w\right)\frac{dw}{|w|} \le b\int_{|w|\ge 1}f_X(w)\frac{dw}{|w|} \le b.
\end{align}

\smallskip

Plugging \eqref{1st int}, \eqref{2nd int} and \eqref{3rd int} into \eqref{int splitting z small}, the result follows for $|z| \le 1$.

\smallskip

Now, if $|z| \ge 1$, then write 
\begin{align}\label{int splitting z large}
\nonumber f_{X\cdot Y}(z) & = \int_{-\infty}^{\infty}f_{X, Y}\left(w, \frac z w\right)\frac{dw}{|w|}
\\ & = \int_{|w| \le |z| }f_X(w)\cdot f_{Y}\left(\frac z w\right)\frac{dw}{|w|} + \int_{|w| \ge |z| }f_X(w)\cdot f_{Y}\left(\frac z w\right)\frac{dw}{|w|}.
\end{align}

\smallskip

For the first term, we have
\begin{align} \label{4th int}
\int_{|w| \le |z| }f_X(w)\cdot f_{Y}\left(\frac z w\right)\frac{dw}{|w|} \le b \int_{|y|\ge 1}f_Y(y)\frac{dy}{|y|} \le b .
\end{align}

\smallskip

And, for the second, by \eqref{3rd int}
\begin{align} \label{5th int}
\int_{|w| \ge |z| }f_X(w)\cdot f_{Y}\left(\frac z w\right)\frac{dw}{|w|} \le \int_{|w| \ge 1 }f_X(w)\cdot f_{Y}\left(\frac z w\right)\frac{dw}{|w|} \le b.
\end{align}

\smallskip

Plugging \eqref{4th int} and \eqref{5th int} into \eqref{int splitting z large}, the result follows.
\end{proof}

\smallskip

Using Proposition \ref{prop prod}, we immediately obtain the following: 

\begin{cor}\label{cor density}
Let $X$ and $Y$ be independent continuous random variables. Then for every $t \in (0, 1)$ and every $\gamma \in \R $, 
\begin{align*}
\p \Big[ |X\cdot Y +\gamma| < t\Big] \le 4bt + 4b^2t(1+|\log t|),
\end{align*}
where $b>0$ is a uniform upper bound on their densities.
\end{cor}

\begin{proof}
Note that the function 
$$ g(z) = \left(2b+2b^2|\log(|z|)|\right)\mathbbm 1_{\{|z|\le1\}}+2b\mathbbm 1_{\{|z|>1\}}$$
satisfies $g(|z_1|) \le g(|z_2|)$ whenever $|z_1| \ge |z_2|$. Thus, we have for every $\gamma \in \R$, $t\in(0,1)$,
\begin{align*}
\int_{\gamma-t}^{\gamma+t}g(z)dz & \le \int_{-t}^{t}g(z)dz = \int_{-t}^t \left(2b+2b^2|\log(|z|)|\right)dz = 4bt +4b^2t(1+|\log t |).
\end{align*}
Thus, by Proposition~\ref{prop prod} we have 
\begin{align*}
\p \Big[ |X\cdot Y-\gamma| <t\Big] & \le \int_{\gamma-t}^{\gamma+t}g(z)dz \le 4bt +4b^2t(1+|\log t |).
\end{align*}
\end{proof}

\smallskip

We also obtain the following corollary.
\begin{cor}
Let $T = \{T_{i, j}\}_{i, j \le 2}$ be a random matrix such that $T_{1, 1}$ and $T_{2, 2}$ are continuous random variables, each independent of all the other entries of $T$. Then for every $t\in(0,1)$
\begin{align*}
\p \Big[{|\det(T)|^{1/2}} \le t\Big] \le 4bt^2 + 4b^2t^2(1+2|\log t|), 
\end{align*}
where $b>0$ is a uniform upper bound on the densities of $T_{1, 1}$, $T_{2, 2}$. 
\end{cor}

\begin{proof}
We have, 
\begin{align*}
\p \Big[ |\det(T)| \le t\Big] & = \p \Big[ |T_{1, 1}\cdot T_{2, 2}- T_{1, 2}\cdot T_{2, 1}| \le t\Big]
\\ & = \E \Bigg[ \p \Big[|T_{1, 1}\cdot T_{2, 2}- T_{1, 2}\cdot T_{2, 1}| \le t \Big| T_{1, 1}, T_{2, 2}\Big]\Bigg]
\\ & \le \sup_{\gamma\in \R } \p \Big[ |T_{1, 1}\cdot T_{2, 2} +\gamma| < t\Big] 
\\ & \le 4bt + 4b^2t(1+|\log t|),
\end{align*}
where in the last inequality we used Corollary~\ref{cor density}. Replacing $t$ by $t^2$, the result follows.
\end{proof}

\smallskip

{\bf Acknowledgements.} We thank Alexander Litvak and Nicole Tomczak-Jaegermann for helpful discussions and comments.

\end{document}